\documentclass{amsart}

\usepackage{amsthm}
\usepackage{amssymb}
\usepackage{bbm}
\usepackage{mathtools}

\newcommand{\x}{\times}
\newcommand{\p}{\partial}

\newcommand{\id}{\mathrm{id}}
\renewcommand{\d}{\mathrm{d}}
\newcommand{\e}{\mathrm{e}}
\newcommand{\irm}{\mathrm{i}}

\newcommand{\Lap}{\mathrm{\Delta}}

\newcommand{\eqnb}{\begin{equation}}
\newcommand{\eqnbs}{\begin{equation*}}
\newcommand{\eqnbsa}{\begin{equation*}\begin{aligned}}
\newcommand{\eqnba}{\begin{equation}\begin{aligned}}
\newcommand{\eqnbl}[1]{\begin{equation}\label{#1}}
\newcommand{\eqnbal}[1]{\begin{equation}\label{#1}\begin{aligned}}
\newcommand{\eqnes}{\end{equation*}}
\newcommand{\eqne}{\end{equation}}
\newcommand{\eqnesa}{\end{aligned}\end{equation*}}
\newcommand{\eqnea}{\end{aligned}\end{equation}}

\newcommand{\re}[1]{(\ref{#1})}
\newcommand{\oneChar}{\hspace{11pt}}
\newcommand{\comment}[1]{}

\newcommand{\RR}{\mathbb{R}}
\newcommand{\TT}{\mathbb{T}}
\newcommand{\PP}{\mathbb{P}}
\newcommand{\ZZ}{\mathbb{Z}}

\newcounter{ThmCount}
\newcounter{MainThmIndex}
\newcounter{TempCounter}

\newtheorem{theorem}[ThmCount]{Theorem}
\newtheorem{proposition}[ThmCount]{Proposition}
\newtheorem{lemma}[ThmCount]{Lemma}
\newtheorem{corollary}[ThmCount]{Corollary}
{\bf}{\it}

\begin{document}
\title[A model for the NS equations using magnetization variables]{On a model for the Navier--Stokes equations using magnetization variables}
\author{Benjamin C. Pooley}
\address {Mathematics Institute, University of Warwick, Coventry, CV4 7AL, UK}
\email{ ben.c.pooley\makeatletter @\makeatother gmail.com}

\begin{abstract}
It is known that in a classical setting, the Navier--Stokes equations can be reformulated in terms of so-called magnetization variables $w$ that satisfy
\eqnbl{Abs_magform}
\p_tw + (\PP w \cdot\nabla)w + (\nabla \PP w)^\top w - \Lap w =0,
\eqne and relate to the velocity $u$ via a Leray projection $u=\PP w$. We will prove the equivalence of these formulations in the setting of weak solutions that are also in $L^\infty(0,T;H^{1/2})\cap L^2(0,T;H^{3/2})$ on the 3-dimensional torus.

Our main focus is the proof of global well-posedness in $H^{1/2}$ for a new variant of \re{Abs_magform}, where $\PP w$ is replaced by $w$ in the second nonlinear term:
\eqnbl{Abs_Simplified}
\p_tw + (\PP w \cdot\nabla)w + \frac{1}{2}\nabla|w|^2- \Lap w =0.
\eqne This is based on a maximum principle, analogous to a similar property of the Burgers equations. 

\end{abstract}

\setcounter{equation}{0}
\maketitle
\section{Introduction}
The 3D Navier--Stokes equations model the flow of an incompressible viscous fluid and comprise the following system:  
\eqnbl{eqNSE1}
\p_tu + (u\cdot\nabla)u- \nu\Lap u +\nabla p = 0,
\eqne
\eqnbl{eqNSE2}
\nabla\cdot u=0,\ u(0,x)=u_0(x).
\eqne
Here the velocity $u(x,t)$ is an unknown evolving vectorfield and $p(x,t)$ is the unknown scalar pressure. The viscosity $\nu>0$ will not play a significant role in our analysis, so we take $\nu=1$ hereafter. 

Global existence of weak solutions $u\in L^\infty(0,T;L^2_\sigma)\cap L^2(0,T;H^1)$ satisfying a certain energy inequality in $\RR^3$ has been known since 1934, due to the fundamental contributions by Leray \cite{Leray_1934}. Since then there has been a great deal of progress, for example in the study of local well-posedness and global existence for small-data in certain critical spaces, as well as a number of important partial regularity results. However, the question of whether a function space exists in which we have global well-posedness for arbitrary initial data remains a major open problem. 

For further discussion of some of the more well-known theory of the Navier--Stokes equations see, for example, \cite{ConstFoias},  \cite{Fefferman_Clay}, \cite{Fujita_Kato_1964}, \cite{LadyzhenskayaFluidBook} and \cite{JCR_NSE_book}.

Given the challenge posed by the global well-posedness problem for this system, it can be useful to consider model problems. In this paper we present a natural model of the Navier--Stokes equations arising from the magnetization-variables formulation via a modification of one of the nonlinear terms that does not affect the scaling of the equations. For this system we can prove a global well-posedness result by virtue of a Burgers-type maximum principle.

The magnetization-variables formulation (also ``Kuzmin-Oseledets'' or ``velicity'' formulation) is more well known in the study of the Euler equations, but in the case of the Navier--Stokes system it has previously been discussed in, for example, \cite{MontSPok_2002} and \cite{Chorin_Book}. Denoting the usual fluid velocity by $u$, this formulation comprises the following system, where $w$ is called the magnetization variable: 
\eqnbl{eqNSEW1}
\p_tw + (u\cdot\nabla)w +(\nabla u)^\top w - \Lap w = 0
\eqne   
\eqnbl{eqNSEW2}
u=\PP w.
\eqne
Here $\PP$ denotes the Leray projection of $L^2$ onto $L^2_\sigma$, the closure of divergence-free functions. Unless stated otherwise, the analysis in this paper will take place under periodic boundary conditions, and the spatial domain will be denoted by \[\TT^3\coloneqq\RR^3/2\pi\ZZ^3.\]

In Section \ref{secDerivRegularity} we will review the equivalence between the two formulations for classical solutions before proving new results about the correspondence in a weak setting. Specifically, in the context of weak solutions on $\TT^3$ (which are defined below), we will show that for a weak solution $w\in L^\infty(0,T;L^2)\cap  L^2(0,T;H^1)$ the projection $\PP w$ is a weak solution of the Navier--Stokes equations. Constructing a solution $w$ from a solution $u$ of the Navier--Stokes equations is less straightforward, however we prove that if $u\in L^\infty(0,T;H^{1/2})\cap L^2(0,T;H^{3/2})$ then there exists a weak solution of (\ref{eqNSEW1}--\ref{eqNSEW2}), with the additional regularity $w\in L^\infty(0,T;H^{1/2})\cap L^2(0,T;H^{3/2})$, for any $w_0\in H^{1/2}$ such that $u_0=\PP w_0$. Moreover, $w$ is unique in this class.

Section \ref{secBurgersVariant} contains our main result. Replacing $u=\PP w$ with $w$ in the second nonlinear term of \re{eqNSEW1}, we arrive at the model system
\eqnbl{eqBurgersW1_pre}
\p_tw+((\PP w)\cdot\nabla)w + \frac{1}{2}\nabla |w|^2 - \Lap w = 0.
\eqne
Noting that this system admits a maximum principle, akin to one that holds for the Burgers equations, we adapt arguments applicable to the latter system (see \cite{P_JCR_Burgers_2015}) to prove the following theorem (note that weak solutions will be defined carefully later). 

\setcounter{MainThmIndex}{\value{ThmCount}}
\begin{theorem}\label{thmH1/2}
Given $w_0\in \dot H^{1/2}(\TT^3)$ there exists a global weak solution $w$ of \re{eqBurgersW1_pre} with $w\in C([0,\infty);\dot H^{1/2})\cap L^2(0,\infty;\dot H^{3/2})$, unique in this class, such that $w(0,x)=w_0(x)$ for all $x\in\TT^3$. Moreover $w\in C^1((0,\infty);C(\TT^3))\cap C((0,\infty);C^2(\TT^3))$ is a classical solution, except at time $t=0$.
\end{theorem}

Our approach to the question of global well-posedness for the Navier--Stokes equations, via the analysis of a reformulation with modified nonlinearity, is in the spirit of other recent work. For example, Chae \cite{Chae_Model_2015} discusses the equivalence between the Navier--Stokes equations and the system
\eqnbs
\left\{
\begin{aligned}
\p_tu +R\x R\x (u\x\omega) -\Lap u=0,\\
\omega=\nabla\x u
\end{aligned}
\right.
\eqnes
where 
\[
R\x u \coloneqq (R_2 u_3 - R_3u_2, R_3u_1-R_1u_3, R_1u_2-R_2u_1)
\] is a combination of the Riesz transforms $R_1$, $R_2$ and $R_3$ on $\RR^3$ applied to $u=(u_1,u_2,u_3)$. He then shows that the simplified system
\eqnbs
\left\{
\begin{aligned}
\p_tu +R\x (u\x\omega) -\Lap u=0,\\
\omega=\nabla\x u
\end{aligned}
\right.
\eqnes
is globally well-posed, in the sense of weak solutions $u\in C([0,T);H^m(\RR^3))\cap L^2(0,T;H^{m+1}(\RR^3))$ for $m>5/2$. 

In contrast, it is shown by Tao \cite{Tao_2014}, that there exists an ``averaged'' version of the classical nonlinear term such that the modified system admits a smooth solution that blows up in finite time.

Throughout this paper we will find estimates using the fractional derivative operator $\Lambda^s$, defined by
\[
\Lambda^s f(x)\coloneqq\sum_{k\in\ZZ^3}|k|^s\hat f_k\e^{\irm k\cdot x}\in L^2(\TT^3),
\]  
for
\[
f(x)=\sum_{k\in\ZZ^3}\hat f_k\e^{\irm k\cdot x} \in H^s(\TT^3).
\]
We also denote the seminorm $\|\Lambda^s\cdot\|_{L^2}$ on $H^s$ by $\|\cdot\|_s$. Sometimes we will use the fact that $\|f\|_s\leq\|f\|_t$ for $0<s\leq t$ and that $\Lambda^2=(-\Lap)$. Note that the Sobolev norm $\|\cdot\|_{H^s}$ defined by
\[
\|f\|_{H^s}\coloneqq \left(\sum_{k\in\ZZ^3}(1+|k|^{2s})|\hat f_k|^2\right)^{1/2}
\]
is equivalent to the norm $\|\cdot\|_{L^2}+\|\cdot\|_{s}$.

\section{Derivation and regularity}\label{secDerivRegularity}
\subsection{Classical solutions}\label{subsecClassicalSolns}
The following propositions show that the systems (\ref{eqNSE1},\ref{eqNSE2}) and 
 (\ref{eqNSEW1})+(\ref{eqNSEW2})  are equivalent for classical solutions on the interior of a domain $\Omega\subseteq\RR^3$ (or on the torus $\TT^3$). The manipulations in the proofs are similar to the derivation of the Weber formula \cite{Weber} for the Euler equations as described by Constantin \cite{Const_EL_Local_2000}, see also \cite{P_JCR_EL_2015}. 
\begin{proposition}\label{propELNS_NS}
If  $u,w\in C^1([0,T];C^2(\Omega))$ satisfy \re{eqNSEW1} and $u=\PP w$ then there exists $p\in C([0,T];C^1(\Omega))$ such that $(u,p)$ is a solution of the Navier--Stokes equations \re{eqNSE1}.
\end{proposition}
\begin{proof}
By the Helmholtz decomposition (see \cite{JCR_NSE_book}, for example), there exists $q\in C^1([0,T];C^3(\Omega))$ such that
\[
u=w-\nabla q.
\]
It is clear that $u$ is divergence free so we must prove that \re{eqNSE1} is satisfied. Indeed we have 
\eqnbsa
\p_tu+(u\cdot\nabla)u -\Lap u&=\p_tw+(u\cdot\nabla)w -\Lap w -\nabla(\p_tq-\Lap q)-(u\cdot\nabla)\nabla q\\
&=w_t+(u\cdot\nabla)w -\Lap w\\
&\oneChar -\nabla(q_t-\Lap q+ (u\cdot\nabla)q +\tfrac{1}{2}|u|^2) + (\nabla u)^\top w\\
&=-\nabla p,
\eqnesa
where $p\coloneqq(q_t-\Lap q+ (u\cdot\nabla)q +\frac{1}{2}|u|^2)$. In the second line we used the commutation relation (see \cite{Const_EL_Local_2000})
\eqnbsa
(u\cdot\nabla)\nabla q = \nabla[(u\cdot\nabla)q] - (\nabla u)^\top \nabla q &=\nabla[(u\cdot\nabla)q] - (\nabla u)^\top (w-u)\\
&=\nabla[(u\cdot\nabla)q] +\frac{1}{2}\nabla |u|^2 - (\nabla u)^\top w.\qedhere
\eqnesa
\end{proof}

\begin{proposition}\label{propNS_ELNS}
If $u\in C^1([0,T];C^2(\Omega))$ and $p\in C([0,T];C^1(\Omega))$ satisfy the Navier--Stokes equations then for any $w_0\in C^2(\Omega)$ such that $\PP w_0=u(0)$, there exists a unique $w\in C^1([0,T];C^2(\Omega))$ such that $u,w$ satisfy \re{eqNSEW1} and $u=\PP w$.
\end{proposition}
\begin{proof}
By standard techniques for parabolic PDEs (see \cite{Evans_PDEBook}, for example) there exists a unique $q\in C^1([0,T];C^3(\Omega))$ such that 
\[
\p_t q +(u\cdot\nabla)q-\Lap q = p-\frac{1}{2}|u|^2
\]
and $q(t,x)=0$ for all $(t,x)\in(\{0\}\x\Omega)\cup ([0,T)\x\p\Omega)$. If we set $w\coloneqq u+\nabla q$ then $u=\PP(w)$ and
\eqnbsa
\p_tw + (u\cdot\nabla) w +(\nabla u)^\top w -\Lap w = \p_tu + (u\cdot\nabla)u + \frac{1}{2}\nabla |u|^2 -\Lap u\\
\oneChar +\nabla \p_tq + (u\cdot\nabla)\nabla q + (\nabla u)^\top\nabla q -\nabla\Lap q\\
= \nabla \left( - p + \frac{1}{2}|u|^2 + \p_tq +(u\cdot\nabla)q-\Lap q \right)=0.
\eqnesa
Hence there exists $w\in C^1([0,T];C^2(\Omega))$ such that $u,w$ satisfy \re{eqNSEW1} and \re{eqNSEW2}.

Uniqueness follows from the fact that any two solutions $w_1$ and $w_2$ differ only by a gradient $\nabla \tilde q$ for some $\tilde q$ that satisfies
\[
\p_t \tilde q +(u\cdot\nabla) \tilde q-\Lap\tilde q = h(t);\, \tilde q(0,x) =C
\] 
for some function $h$ that is independent of $x$, and some constant $C$. Hence $\tilde q$ depends only on time, and $\nabla q\equiv 0$.
 \end{proof}

\subsection{Partial equivalence for weak solutions}

Proposition \ref{propELNS_NS} can be strengthened to apply to weak solutions of \re{eqNSEW1} and \re{eqNSEW2}. We say that $w\in L^\infty(0,T;L^2)\cap L^2(0,T;H^1)$ is a weak solution of \re{eqNSEW1} for initial data $w_0\in L^2(\TT^3)$ if for all $\phi\in C^{\infty}_c([0,T)\x\TT^3)$ and all $t\in[0,T)$
\eqnbal{eqWeakNSEW}
 (w(t),\phi(t))_{L^2}+\int_0^t ((\PP w\cdot\nabla)w + (\nabla\PP w)^\top w,\phi)_{L^2} +(\nabla w,\nabla\phi)_{L^2}\,\d s\\
=(w_0,\phi(0))_{L^2}+\int_0^t(w(s),\p_t\phi(s))_{L^2}\,\d s.
\eqnea
We call a weak solution $w$ an $H^{1/2}$-solution if it has the additional regularity $w\in L^\infty(0,T; H^{1/2})\cap L^2(0,T; H^{3/2})$.

For the Navier--Stokes equations we use a similar definition of weak solutions: $u\in L^\infty(0,T;L^2)\cap L^2(0,T;H^1)$ is a weak solution to \re{eqNSE1} corresponding to the initial data $u_0\in L^2(\TT^3)$ if for all divergence-free test functions $\psi\in C^\infty_c([0,T)\x\TT^3)$ with $\nabla\cdot\psi\equiv 0$ we have
\begin{multline}\label{eqNS_wk}
(u(t),\psi(t))_{L^2}+\int_0^t((u\cdot\nabla) u,\psi)_{L^2} + (\nabla u,\nabla \psi)_{L^2}\d s\\
 = (u_0,\psi(0))_{L^2} + \int_0^t(u(s),\p_t\phi(s))_{L^2}\,\d s
\end{multline}
for all $t\in[0,T)$.

The proof of Proposition \ref{propELNS_NS} can easily be adapted to the setting of weak solutions, giving the following result:
\begin{proposition}\label{propELNS_NS_wk}
Suppose that $w\in L^\infty(0,T;L^2)\cap L^2(0,T;H^1) $ is a weak solution of \re{eqNSEW1} for initial data $w_0\in L^2(\TT^3)$. Then $u\coloneqq \PP w$ is a weak solution of the Navier--Stokes equations for initial data $u_0=\PP w_0$.
\end{proposition}  
The details of the proof are not difficult and are omitted. We note only that it suffices to show that if $v\in H^1(\TT^3)$ then for all $\psi\in C^\infty_c(\TT^3)$ with $\nabla\cdot\psi =0$ we have
\eqnbl{propELNS_NS_wk:eq1}
((\PP v\cdot\nabla)v+(\nabla\PP v)^\top v,\psi)_{L^2}=((\PP v\cdot\nabla)\PP v,\psi)_{L^2}.
\eqne 
For smooth functions this follows from the proof of Proposition \ref{propELNS_NS} and the full justification is not difficult.
	Another consequence of \re{propELNS_NS_wk:eq1} is the following partial converse.
\begin{corollary}
If $u\in L^\infty(0,T;L^2)\cap L^2(0,T;H^1) $ is a weak solution of the Navier--Stokes equations then, for any $w\in  L^\infty(0,T;L^2)\cap L^2(0,T;H^1) $ such that $\PP w=u $, $w$ satisfies \re{eqWeakNSEW} for all test functions $\phi\in C^\infty_c([0,T)\x\TT^3)$ that are divergence free. 
\end{corollary}
Note that this does not imply that $w$ is a weak solution of \re{eqNSEW1}, since in the definition we allowed test functions with non-zero divergence.

\subsection{Equivalence for $H^{1/2}$-solutions}\label{subsecStrongSolns}
In  this section we will show that a weak solution of the Navier--Stokes equations with the regularity $u\in L^\infty(0,T;H^{1/2})\cap  L^2(0,T;H^{3/2}))$ corresponds to a unique $H^{1/2}$-solution $w$ of \re{eqNSEW1} such that $\PP w=u$, subject to fixing $w_0$ with $\PP w_0=u_0$. 

We begin by stating a well-posedness result for the linear system
\eqnbl{eqNSEW1_fixedU}
\p_tw+(u\cdot\nabla)w+(\nabla u)^\top w-\Lap w=0,
\eqne
where $u\in L^\infty(0,T;H^{1/2})\cap L^2(0,T;H^{3/2})$ is regarded as a fixed function, which we need not assume to be divergence free. 

Given $u$ as above, we say that $w$ is an $H^{1/2}$-solution of \re{eqNSEW1_fixedU} if $w\in L^\infty(0,T;H^{1/2})\cap L^2(0,T;H^{3/2})$ and is a weak solution, in the sense that
\begin{multline}\label{eqFixedU_Wweak}
(w(t),\phi(t))_{L^2}+\int_0^t((u\cdot\nabla)w(s) + (\nabla u)^\top w(s),\phi(s))_{L^2}\,\d s\\
=(w_0,\phi(0))_{L^2} + \int_0^t(w(s),\p_t\phi(s))\,\d s-\int_0^t(\nabla w(s),\nabla\phi(s))_{L^2}\d s
\end{multline}  
for all test functions $\phi\in C^\infty_c([0,T)\x\TT^3)$.

\begin{proposition}\label{propNSEW1_fixedU}
Fix $u\in L^\infty(0,T;H^{1/2})\cap  L^2(0,T;H^{3/2})$ and $w_0\in H^{1/2}$.  There exists a unique $H^{1/2}$-solution $w$ to \re{eqNSEW1_fixedU}. 
\end{proposition}
This can be proved using a similar approach to our work in the next section, moreover some of the arguments are simpler because this is a linear problem. We therefore omit most of the details.

One noteworthy technicality is that the system \re{eqNSEW1_fixedU} may not  conserve momentum, so we must carefully control the zeroth Fourier coefficient in the case of periodic boundary conditions. This can be dealt with following the example of the diffusive Burgers equations \cite{P_JCR_Burgers_2015}, i.e. by justifying an a priori estimate of the form
\[
|\hat w_0(t)|\leq |\hat w_0(0)|+2\int_0^t\|w\|_{1/2}\|u\|_{1/2}.
\] 
A complete proof of Proposition \ref{propNSEW1_fixedU}  can be found in \cite{Pooley_PhD}.

We now show that if we additionally assume that $u$ is a weak solution of the Navier--Stokes equaitions then $u=\PP w$, subject to choosing the initial data $w_0$ such that $\PP w_0=u_0$.

\begin{proposition}\label{propNSE_W_strong}If $u\in L^\infty(0,T;H^{1/2})\cap L^2(0,T;H^{3/2})$ is a weak solution of the  Navier--Stokes equations and $\PP w_0 =u_0$ then the corresponding $H^{1/2}$-solution $w$ of \re{eqNSEW1_fixedU} satisfies $\PP w = u$, i.e. $(u,w)$ satisfy \re{eqNSEW1} and \re{eqNSEW2} in a weak sense.
\end{proposition}
\begin{proof}
Suppose that $u\in  L^\infty(0,T;H^{1/2})\cap L^2(0,T;H^{3/2})$ is a weak solution  of the Navier--Stokes equations and let $w\in L^\infty(0,T;H^{1/2})\cap L^2(0,T;H^{3/2})$ be the corresponding $H^{1/2}$-solution of \re{eqNSEW1_fixedU}. Then $v\coloneqq\PP w$ satisfies \re{eqFixedU_Wweak}, for all $t\in[0,T)$ and all $\phi\in C^\infty_c([0,T)\x\TT^3)$ such that $\nabla\cdot\phi=0$. Note that the requirement that test functions be divergence free means that $v$ is not necessarily an $H^{1/2}$-solution of \re{eqNSEW1_fixedU}.

To see that $v$ satisfies \re{eqFixedU_Wweak}, suppose that $u$ and $w$ are smooth. We can then write $v=w-\nabla q$ and treat the nonlinear terms as follows: 
\eqnbsa
((u\cdot\nabla)w + (\nabla u)^\top w,\phi) = ((u\cdot\nabla)v + (\nabla u)^\top v + (u\cdot\nabla)\nabla q + (\nabla u)^\top\nabla q,\phi)\\
= ((u\cdot\nabla)v + (\nabla u)^\top v,\phi)\,.
\eqnesa A simple density argument shows that this also holds for $w\in H^1$. 

Now one can prove that $H^{1/2}$-solutions of \re{eqNSEW1_fixedU} are unique when test functions are taken to be divergence free, just as we have uniqueness in Proposition \ref{propNSEW1_fixedU}. Furthermore, we deduce that $u$ is such a solution, from the hypothesis that $u$ is a weak solution of the Navier--Stokes equations (with the specified regularity). Indeed, using the substitution $w=u$, the nonlinear terms in  \re{eqFixedU_Wweak} become
\[
((u\cdot\nabla)u + \tfrac{1}{2}\nabla|u|^2,\phi)_{L^2}=((u\cdot\nabla)u,\phi)_{L^2},
\]  
as $\nabla\cdot\phi=0$. Since $v_0=\PP w_0 = u_0$, it follows from uniqueness that $u=v=\PP w$ as claimed. 
\end{proof}
\section{Global well-posedness for a model system}\label{secBurgersVariant}

So far, we have considered the magnetization variables in the Navier--Stokes equations and proved the equivalence of the formulations for sufficiently regular weak solutions. Due to this equivalence, we do not expect that revisiting this reformulation will quickly yield new information about the Navier--Stokes equations. However, in this section we prove our main result (Theorem \ref{thmH1/2}) that gives global well-posedness for a slight modification of the reformulation.

Recall that the equations satisfied by the magnetization variables are
\[
\p_tw + ((\PP w)\cdot\nabla)w +(\nabla\PP w)^\top w -\Lap w=0.
\]
We will consider the following simplification, obtained by replacing $\PP w$ with $w$ in the second nonlinear term:
\eqnbl{eqBurgersW1}
\p_tw+((\PP w)\cdot\nabla)w + \frac{1}{2}\nabla |w|^2 - \Lap w = 0.
\eqne
These resemble the 3D viscous Burgers equations,
 which are a classical variant of the Navier--Stokes equations obtained by removing the pressure term and the incompressibility constraint:
\eqnbl{eqBurgers}
\p_tw+(w\cdot\nabla)w - \Lap w = 0.
\eqne

In \cite{P_JCR_Burgers_2015}, we showed that on the torus $\TT^3$ and for initial data in $w_0\in H^{1/2}(\TT^3)$ equation \re{eqBurgers} admits a unique strong global solution which is classical for $t>0$. We will show that similar methods apply to \re{eqBurgersW1}, which are closer to the Navier--Stokes equations in the sense that the nonlinear terms would have to be altered more significantly to obtain the Burgers equations. Moreover we will see that unlike solutions of the Burgers equations, solutions of \re{eqBurgersW1} have constant momentum -- a property shared with solutions of the Navier--Stokes equations.

We divide the proof of well-posedness for \re{eqBurgersW1} into two parts: first we prove the global well-posedness of weak solutions  for initial data $w_0\in H^1$; then show a local well-posedness result for $w_0\in H^{1/2}$ that combines with the $H^1$ result to give global well-posedness in this case. The local existence result in $H^{1/2}$ is based on the approach described in \cite{Marin-Rubio_JCR_Sad_2013}, for the Navier--Stokes equations. 

The main estimates we use to deduce global well-posedness arise from the following maximum principle for classical solutions of \re{eqBurgersW1} that we have adapted from \cite{Kiselev_Ladyzhenskaya}.
\begin{lemma}\label{lemMP}
If $w$ is a classical solution of \re{eqBurgersW1} on a time interval $[a,b]$ then
\eqnbl{lemMP:bound1}
\sup_{t\in[a,b]}\|w(t)\|_{L^\infty}\leq \|w(a)\|_{L^\infty}.
\eqne
\end{lemma}
\begin{proof}
Fix $\alpha>0$ and let $v(t,x)\coloneqq \e^{-\alpha t}w(x,t)$ for all $x\in\TT^3$. Then $|v|^2$ satisfies the equation
\eqnbl{lemMP:vEqn}
\p_t |v|^2 + 2\alpha |v|^2 + \nabla |v|^2 w +(\PP w)\cdot\nabla |v|^2 - 2 v\cdot\Lap v=0 .
\eqne
Since $2v\cdot\Lap v = \Lap |v|^2 - 2|\nabla v|^2$ we see that if $(x,t)\in(a,b]\x \TT^3$ is a local maximum of $|v|^2$, then the left-hand side of \re{lemMP:vEqn} is positive unless $|v(x,t)|=0$. Hence
\[ 
\|w(t)\|_{L^\infty}\leq \e^{\alpha t}\|w(a)\|_{L^\infty}.
\]
Now \re{lemMP:bound1} follows because $\alpha>0$ was arbitrary. 
\end{proof}

As we saw in the previous section, solutions of \re{eqNSEW1} for fixed $u$, do not necessarily have constant momentum (a similar technicality occurs for the Burgers equations, see \cite{P_JCR_Burgers_2015}). However, in the case of \re{eqBurgersW1}, like the Navier--Stokes equations, initial data with zero average gives rise to solutions that also have this property for positive times. To see this formally, we integrate \re{eqBurgersW1} over $\TT^3$:
\[
\frac{\d }{\d t} \int_{\TT^3} w\  \d x= -\int_{\TT^3}((\PP w)\cdot\nabla)w + \frac{1}{2}\nabla|w|^2 - \Lap w\,\d x =0
\]  
where the first term on the right-hand side vanishes because $\PP w$ is weakly divergence free and the other terms vanish by periodicity. For this reason, in what follows, we will prove well-posedness for solutions in certain homogeneous Sobolev spaces $\dot H^s(\TT^3)$.

As in the previous section we will at first consider a weak formulation of \re{eqBurgersW1}. We call $w\in L^\infty (0,T;L^2)\cap L^2(0,T;\dot H^1)$ a weak solution of \re{eqBurgersW1} with initial data $w_0\in L^2$ if 
\begin{multline}\label{eqBurgersW1Wk}
\int_0^t(((\PP w)\cdot\nabla)w(s)+ (\nabla w)^\top w,\phi(s))_{L^2}+(\nabla w(s),\nabla\phi(s))_{L^2}\,\d s\\
 =(w_0,\phi(0))_{L^2}-(w(t),\phi(t))_{L^2}+\int_0^t(w(s),\p_t\phi(s))_{L^2}\,\d s
\end{multline}
for all $\phi\in C^\infty_c([0,T)\x\TT^3)$ and all $t\in[0,T)$. 

We have not been able to find weak solutions of \re{eqBurgersW1} directly, as we would for the Navier--Stokes equations. Indeed, the second nonlinear term does not seem amenable to the necessary energy estimates if we only have $w_0\in L^2$. However for $w_0\in \dot H^{1/2}$ we will show that there exists a unique weak solution with the additional regularity $w\in L^\infty(0,T;\dot H^{1/2})\cap L^2(0,T;\dot H^{3/2})$ for some $T>0$. We call solutions that are at least this regular $H^{1/2}$-solutions. Moreover, we will show that the solutions become smooth, immediately after the initial time, and can be extended to solutions on $[0,\infty)$.  

Before recalling the main statement and beginning the proof we define the Galerkin approximations that will be used in this section. For fixed $w_0\in \dot H^{1/2}$ we denote by $w_n\in C^\infty([0,T_n]\x\TT^3)$ the solution of the truncated equation
\eqnbl{eqBurgersW1Galerkin}
\p_tw_n +P_n\left[((\PP w_n)\cdot\nabla)w_n + \frac{1}{2}\nabla |w_n|^2\right] - \Lap w_n = 0
\eqne  
with initial data $P_n w_0$. Here $T_n>0$ is the maximal existence time for the solution $w_n$, of this system of quadratic ODEs. Here $P_n$ denotes a Fourier truncation:
\[
P_n\left(\frac{1}{(2\pi)^{3/2}}\sum_{k\in\ZZ^3}\hat f(k)\e^{\irm x\cdot k}\right) = \frac{1}{(2\pi)^{3/2}}\sum_{|k|\leq n}\hat f(k)\e^{\irm x\cdot k}.
\]

\setcounter{TempCounter}{\value{ThmCount}}
\setcounter{ThmCount}{\value{MainThmIndex}}
\begin{theorem}
Given $w_0\in \dot H^{1/2}(\TT^3)$ there exists a unique global $H^{1/2}$-solution of \re{eqBurgersW1} $w\in C([0,\infty);\dot H^{1/2})\cap L^2(0,\infty;\dot H^{3/2})$ such that $w(0,x)=w_0(x)$ for all $x\in\TT^3$. Moreover $w\in C^1((0,\infty);C(\TT^3))\cap C((0,\infty);C^2(\TT^3))$ is a classical solution, except at time $t=0$.
\end{theorem}
\setcounter{ThmCount}{\value{TempCounter}}

This result is a consequence of the following two theorems.
\begin{theorem}\label{thmH1}
If $w_0\in \dot H^1(\TT^3)$ there exists a unique global solution of \re{eqBurgersW1} $w\in C([0,\infty);\dot H^1)\cap L^2(0,T;\dot H^{2})$ such that $w(0)=w_0$. Moreover $w$ is a classical solution, except possibly at time $t=0$.   
\end{theorem}
In the case of initial data in $\dot H^1$, we will obtain local well-posedness and smoothness in the same way as we can for the Navier--Stokes equations. Global well-posedness then follows, using estimates based on the maximum principle (Lemma \ref{lemMP}).

\begin{theorem}\label{thmH1/2loc}
For any $w_0\in \dot H^{1/2}$ there exists a unique $H^{1/2}$-solution, $w$, of \re{eqBurgersW1} on $[0,T)$ for some $T>0$. In addition, $w\in C([0,T);H^{1/2})$.
\end{theorem}   
The proof of this follows the method of \cite{Calderon_1990} (See also \cite{Chemin_book_2006} or \cite{Marin-Rubio_JCR_Sad_2013} for expositions) in which we decompose the equations into a heat part and a nonlinear part with vanishing initial data.
\subsection{Proof of Theorem \ref{thmH1}}
First, note that if $w_n$ satisfies \re{eqBurgersW1Galerkin} then, by the arguments above, 
\[
\int_{\TT^3}w_n(x,t)\,\d x =\int_{\TT^3}P_n w_0 (x)\,\d x=0.
\] 
Integrating \re{eqBurgersW1Galerkin} against $2\Lambda^2 w_n$, and proceeding as for strong solutions of the Navier--Stokes equations (see \cite{JCR_NSE_book}, for example), yields
\eqnbl{eqBurgersW1H1Est1}
\frac{\d }{\d t}\|w_n\|_{1}^2 + \|w_n(t)\|_{2}^2 \leq c\|w_n(t)\|_{1}^6 
\eqne
for all $t\in[0,T_n]$ and some $c>0$. Considering only the terms in $\|w_n\|_1^2$ and solving the resulting differential inequality, we obtain
\[
\|w_n(t)\|_1^2\leq\frac{\|P_nw_0\|_1^2}{\sqrt{1-2ct\|P_nw_0\|_1^4}}.
\] 
Fixing $T<(2c\|w_0\|_1^4)^{-1}$, it follows from maximality of $T_n$ that $T_n>T$ and $\|w_n(t)\|_1$ is bounded, independent of $n$, on $[0,T)$. From \re{eqBurgersW1H1Est1}, it then follows that $w_n$ is uniformly bounded in $L^2(0,T;\dot H^{2})$.

Using these uniform bounds we have the following bounds on the nonlinear terms from \re{eqBurgersW1Galerkin} in $L^2(0,T;L^2)$:
\eqnbsa
\left(\int_{\TT^3}|(\nabla w_n)^\top w_n|^2\right)^{1/2}&\leq \left(\int_{\TT^3} |\nabla w_n|^3\right)^{1/3}\left(\int_{\TT^3}|w_n|^6\right)^{1/6}\\
&\leq c\|w_n\|_{H^{3/2}}\|w_n\|_{H^1}\in L^2(0,T).
\eqnesa
A similar estimate holds for the other nonlinear term, hence $\p_tw_n$ is uniformly bounded in $L^2(0,T;L^2)$. By the Aubin--Lions lemma, there exists a subsequence relabelled $w_n\to w\in L^2(0,T;\dot H^1)$ such that $w$ is a weak solution of \re{eqBurgersW1Wk} and also $\p_tw\in L^2(0,T;L^2)$, $w\in C([0,T);\dot H^1)\cap L^2(0,T;\dot H^2)$. 

To prove that $w$ is a classical solution after the initial time, we have the following lemma. We omit the proof because it is very similar to arguments applicable to the Navier--Stokes equations which are described in \cite{ConstFoias} and \cite{JCR_2006}, for example.

\begin{lemma}\label{lemBootstrap}
If the approximations $w_n$ are uniformly bounded in $L^2(\varepsilon,T;\dot H^{s})$ for $s>3/2$ and some $\varepsilon\geq 0$ such that $\|w_n(\varepsilon)\|_{s}<\infty$, then they are also bounded uniformly in $L^\infty(\varepsilon,T;\dot H^s)\cap L^2(\varepsilon,T;\dot H^{s+1})$.
\end{lemma}
Applying this lemma five times, we see that $(w_n)_{n=1}^\infty$ is a bounded sequence in $L^\infty(\varepsilon,T;\dot H^6)$ for all $\varepsilon\in(0,T)$.  Using the Banach algebra property of $H^s$ for $s>3/2$, this gives us the following estimates on the time derivatives of $w_n$:
\[
\sup_{t\in(\varepsilon,T)}\left\|\p_t w_n(t)\right\|_4\leq c \sup_{t\in(\varepsilon,T)}(\|w_n(t)\|_4\|w_n(t)\|_5 + \|w_n(t)\|_6)
\]
and (differentiating \re{eqBurgersW1Galerkin})
\eqnbsa
\sup_{t\in(\varepsilon,T)}\left\|\p_t^2 w_n(t)\right\|_2\leq c\sup_{t\in(\varepsilon,T)}\left(\left\|\p_t w_n(t)\right\|_4 + \left\|\p_t w_n(t)\right\|_2\|w_n(t)\|_3\right.\\
+\left.\left\|\p_t w_n(t)\right\|_3\|w_n(t)\|_2\right).
\eqnesa
Therefore $w_n$ is uniformly bounded in $H^2(\varepsilon,T;\dot H^2)\cap H^1(\varepsilon,T;\dot H^4)$. This regularity passes to the limit; hence by Sobolev embeddings $w\in C^1(0,T;C(\TT^3))\cap C(0,T;C^2(\TT^3))$ is a classical solution on $[\varepsilon,T]$. Note that we may consider a closed interval by using the above argument on a larger open interval.   

Since $w$ is a classical solution we can apply Lemma \ref{lemMP} to obtain
\[
\sup_{t\in [\varepsilon,T]}\|w(t)\|_{L^\infty}\leq \|w(\varepsilon)\|_{L^\infty}.
\] 
This allows the following additional $H^1$ estimate:
\eqnbl{eqBurgersW1Linftyapp1}
\frac{\d}{\d t}\|w\|_1^2\leq |(2((\PP w)\cdot\nabla)w + \nabla |w|^2,-\Lap w)_{L^2}| - 2\|w\|_2^2\leq c\|w\|_{L^\infty}^2\|w\|_1^2. 
\eqne
Notice that care must be taken with the first nonlinear term because $\PP w$ is an unbounded operator on $L^\infty$. We therefore argue using the anti-symmetry, 
\[(\PP w\cdot\nabla v_1,v_2)_{L^2}=-(\PP w\cdot\nabla v_2,v_1)_{L^2},\] as follows:
\begin{multline*}
((\PP w\cdot\nabla)w,-\p_{xx} w)_{L^2} =(\p_x[(\PP w\cdot\nabla)w],\p_x w)_{L^2} \\
= ((\PP\p_x w\cdot\nabla)w,\p_x w)_{L^2} + ((\PP w\cdot\nabla \p_x w),\p_x w)_{L^2}
= -((\PP\p_x w\cdot\nabla)\p_x w, w)_{L^2},
\end{multline*}
for any spatial derivative $\p_x$. Hence the inequality
\[
|((\PP w)\cdot\nabla)w,-\Lap w)_{L^2}|\leq \|w\|_1\|w\|_2\|w\|_{L^\infty},
\]
holds, in the absence of $L^\infty$ bounds on $\PP w$.

From \re{eqBurgersW1Linftyapp1} and Lemma \ref{lemMP}, it follows that for all $t\in[0,T)$
\[
\|w(t)\|_{1}^2\leq \|w_0\|^2_1\e^{ct\|w_0\|_{L^\infty}^2}.
\] 
This rules out the finite-time blowup of $\|w(t)\|_1$, therefore since we can extend a solution on $[0,T)$ onto $[0,T+\delta)$ where $\delta \propto \|w(T)\|_1^{-4}$, there exists a solution $w\in C^1([0,\infty);C(\TT^3))\cap C([0,\infty);C^2(\TT^3)) $.   

We have now proved that for initial data in $\dot H^1$ there exists a global weak solution to \re{eqBurgersW1} that is classical, except possibly at the initial time. To complete the proof of Theorem \ref{thmH1} it remains to show that these solutions are unique. The following lemma also shows that even less regular solutions are unique and will be useful in the next section.
\begin{lemma}\label{lemBurgersW1Uniq}
If $w_1$, $w_2\in L^\infty(0,T;\dot H^{1/2})\cap L^2(0,T;\dot H^{3/2})$ are $H^{1/2}$-solutions of \re{eqBurgersW1Wk} corresponding to the same initial data $w_0\in \dot H^{1/2}$ then $w_1=w_2$. 
\end{lemma}  
\begin{proof}
Let $(\psi_n)_{n=1}^\infty$ be a sequence of spatially-periodic test functions in $\psi_n\in C_c^\infty([0,T)\x\TT^3)$ such that $\int_{\TT^3}\psi_n(t)=0$ for all $t\in[0,T)$ and $\psi_n\to w_1-w_2$ in $L^2(0,T;\dot H^{3/2})$. Set $\phi_n \coloneqq \Lambda^1\psi_n$ in \re{eqBurgersW1Wk}, then the difference $w_1-w_2$ satisfies
\eqnbsa
(\Lambda^{1/2}(w_1-w_2),\Lambda^{1/2}\psi_n)_{L^2}+\int_0^t(\Lambda^{1/2}\nabla(w_1-w_2),\Lambda^{1/2}\nabla \psi_n)\\
-\int_0^t(\Lambda^{1/2}(w_1-w_2)(s), \p_t\Lambda^{1/2}\psi_n(s))\,\d s\\
\leq c\int_0^t(\|w_1-w_2\|_{1/2}\|w_1\|_{3/2} +\|w_2\|_{1}\|w_1-w_2\|_1)\|\psi_n\|_{3/2}\,\d s\\
+ c\int_0^t(\|w_1-w_2\|_{1}\|w_1\|_{1} +\|w_2\|_{3/2}\|w_1-w_2\|_{1/2})\|\psi_n\|_{3/2}\,\d s,
\eqnesa
for all $t\in[0,T)$ and every $n$. 
 Hence, letting $n\to\infty$, and applying Young's inequality in the usual way, we see that 
\[
\|w_1-w_2(t)\|_{1/2}^2\leq c\int_0^t(\|w_1\|_{3/2}^2+\|w_2\|_{3/2}^2 + \|w_1\|_1^4+ \|w_2\|_1^4)\|w_1-w_2\|_{1/2}^2\,\d s 
\]
for almost all $t\in[0,T)$. Since the parenthesised part of the integral is in $L^1(0,T)$, Gronwall's Lemma now implies that $\|w_1-w_2(t)\|_{1/2}=0$ for all $t\in[0,T)$. 
 \end{proof}
\subsection{Proof of Theorem \ref{thmH1/2loc}}
In this section we prove the local well-posedness of \re{eqBurgersW1} with initial data $w_0\in\dot H^{1/2}$. Uniqueness follows from Lemma \ref{lemBurgersW1Uniq}, so it suffices to prove local existence of $H^{1/2}$-solutions.

Following \cite{Calderon_1990}, \cite{Marin-Rubio_JCR_Sad_2013}, and \cite{P_JCR_Burgers_2015}, we find the necessary estimates by decomposing the Galerkin approximations $w_n$, which solve \re{eqBurgersW1Galerkin}, into a sum $w_n=v_n+z_n$ where
\eqnbs
\left\{
\begin{aligned}
\p_tv_n-\Lap v_n =0\\
v_n(0)=P_nw_0
\end{aligned}
\right.
\eqnes
and
\eqnbl{eqBurgersW1Calderon1}
\left\{
\begin{aligned}
\p_tz_n-\Lap z_n = -P_n\left[((\PP w_n)\cdot\nabla) w_n + \frac{1}{2}\nabla |w_n|^2\right]\\
z_n(0)=0.
\end{aligned}
\right.
\eqne

From the heat equation satisfied by $v_n$, it is easy to check that for any $t\geq0$ and any $n$
\eqnbl{eqHEEstimate}
\|v_n(t)\|_{1/2}^2 +2\int_0^t\|v_n(s)\|^2_{3/2}\,\d s \leq \|P_n w_0\|^2_{1/2},
\eqne
hence $v_n\in L^\infty(0,T;H^{1/2})\cap L^2(0,T;H^{3/2})$ is uniformly bounded (independent of $n$ and $t$). It therefore suffices to find estimates on $z$ in the same spaces.

Integrating \re{eqBurgersW1Calderon1} against $\Lambda z_n$ yields
\eqnbsa
\frac{1}{2}\frac{\d}{\d t}\|z_n(t)\|_{1/2}^2 + \|z_n(t)\|_{3/2}^2&\leq|(((\PP w_n)\cdot\nabla) w_n+(\nabla w_n)^\top w_n,\Lambda z_n)_{L^2}|\\
&\leq c\|w_n(t)\|_{1}^2\|z_n(t)\|_{3/2}\\
&\leq  2c\|v_n\|_1^4 +\frac{1}{4}\|z_n\|^2_{3/2} + \|z_n\|^2_{3/2}\|z_n\|_{1/2}\\
&\leq 2c\|v_n\|_1^4 +\frac{3}{4}\|z_n\|^2_{3/2} + \frac{1}{2}\|z_n\|^2_{3/2}\|z_n\|_{1/2}^2.
\eqnesa

This can be re-arranged to give a differential inequality of the form 
\[
\frac{\d x}{\d t}+y\leq 2xy+\delta(t)
\]
where $x=\|z_n\|_{1/2}^2$, $y=\tfrac{1}{2}\|z_n\|_{3/2}^2$ and $\delta_n(t)=4c\|v_n(t)\|_1^4$. It follows (see Lemma 10.3 of \cite{JCR_NSE_book}) that $\|z_n\|_{L^\infty(0,T_n;H^{1/2})}\leq 1/4$ and $\|z_n\|_{L^2(0,T_n;H^{3/2})}\leq 1/2$ when
\eqnbl{eqTn}
\int_0^{T_n}\delta_n(t)\ \d t\leq 1/8.
\eqne
Now since $\|v_n(t)\|_1\leq \|v(t)\|_1$ for all $t\geq 0$, we can choose $T>0$ such that \re{eqTn} holds with $T_n=T$ for all $n$. For this $T$ we have uniform bounds on $z_n$ in $L^\infty(0,T;H^{1/2})$ and $L^2(0,T;H^{3/2})$.

It follows that $w_n$ is bounded in $L^\infty(0,T;H^{1/2})\cap L^2(0,T;H^{3/2})$ independent of $n$. A simple argument now yields bounds on $\tfrac{\p}{\p t} w_n\in L^2(0,T;H^{-1/2})$, independent of $n$. Therefore, by the Aubin--Lions lemma, passing to a subsequence we may assume that $w_n$ converges in $L^2(0,T;H^{1/2})$ to a limit $w$ that is an $H^{1/2}$-solution of \re{eqBurgersW1}. 

The fact that $w\in C([0,T);H^{1/2})$ follows from the embedding 
\[
\{f\in L^2(0,T;H^{3/2})\colon \p_t f\in L^2(0,T;H^{-1/2})\}\hookrightarrow C([0,T);H^{1/2}),
\]
see Chapter 7 of \cite{Roubicek_book}, for example. The definition of weak solutions ensures that the continuous representative attains the initial data.  This completes the proof of Theorem \ref{thmH1/2loc}.

 Since for all $\varepsilon>0$ there exists $t\in(0,\varepsilon)$ such that $w(t)\in \dot H^1$, Theorem \ref{thmH1} implies that $w$ is a classical solution on $(0,T)$ and can be extended to a classical solution on $(0,\infty)$. By Lemma \ref{lemBurgersW1Uniq} this solution is unique. Thus Theorem \ref{thmH1/2} is proved.  
\section{Conclusions}

We have presented a new model system for the Navier--Stokes equations, obtained from the magnetization variables formulation by the omitting a Leray projector in one of the nonlinear terms. Like the Burgers equations, the new system has a maximum principle, and, additionally, like the Navier--Stokes equations exhibits conservation of momentum. Using these observations, we showed that the system admits global-in-time existence and uniqueness of solutions in $\dot H^{1/2}(\TT^3)$. 

As for the Navier--Stokes equations,  $\dot H^{1/2}(\TT^3)$ is a critical space with repect to the natural scaling of the model equations. However, it is not clear whether one can obtain existence results for the latter system in sub-critical spaces, in particular $L^2$, as the modified nonlinear term does not admit the cancellations by which the necessary energy estimates are usually obtained. 

A solution of the Navier--Stokes equations (in magnetization variables \re{eqNSEW1}) satisfies an inhomogeneous version of the model system where the right-hand side depends on $\nabla(w-\PP w)$. Using this fact one can follow the analysis above to arrive at a necessary criterion for blowup of the Navier--Stokes equations, which to our knowledge is new. Indeed, suppose that $w$ is a classical solution of
\begin{equation}\label{eqConcInhom}
\p_t w + (\PP w\cdot\nabla)w +\frac{1}{2}\nabla |w|^2-\Lap w=f(x,t)
\end{equation}
for $f\in C^1([0,T]\x\TT^3)$. In particular \re{eqConcInhom} is equivalent to the Navier--Stokes equations when $f=(\nabla(w-\PP w))^\top w$. Now, following the proof of Lemma \ref{lemMP} (see also \cite{Kiselev_Ladyzhenskaya}) we obtain 
\begin{equation}\label{eqMPInhom}
\|w(t)\|_{L^\infty}\leq \frac{1}{\alpha}\e^{\alpha t}\sup_{s\leq t}\|\e^{-\alpha s}f(s)\|_{L^\infty}+\e^{\alpha t}\|w_0\|_{L^\infty}
\end{equation}
for any $\alpha >0$.  An estimate similar to \re{eqBurgersW1Linftyapp1} also yields
\[
\|w(t)\|_1^2\leq \|w_0\|_1^2\e^{c\int_0^t\|w\|_{L^\infty}^2}+c\int_0^t\|f(s)\|_{L^2}^2\e^{c\int_s^t\|w\|_{L^\infty}^2}\ \d s.
\]
It follows that if $\|u\|_1=\|\PP w\|_1$ becomes unbounded at time $T^\ast>0$ then  
\[
\int_{t_0}^{T^\ast}\sup_{s\in[t_0,t]}\|(\nabla(w-\PP w))^\top w(s)\|^2_{L^\infty(\TT^3)}\ \d t=\infty,
\] where $w$ is any solution of \re{eqNSEW1} given by Proposition \ref{propNSE_W_strong}, corresponding to $u$ on a time interval $[t_0,T^\ast)$. It is worth noting that $w(t_0)=u(t_0)$ (by construction), so $\nabla(w-\PP w)(t_0)=0$, for each such solution $w$.

Another question that merits further investigation is whether interesting one-parameter families of systems can be constructed to interpolate between the model and the classical system. For example, one might consider systems of the form
\[
\p_tw +(\PP w\cdot\nabla)w+ (\nabla A_\lambda w)^\top w -\Lap w=0
\]
where $A$ is a family of operators such that $A_0=\id$ and $A_1=\PP$. An obvious example is $A_\lambda=\lambda \PP + (1-\lambda)\id$. In this case, a calculation similar to that in Proposition \ref{propELNS_NS} shows that the corresponding system for the velocity $u_\lambda$ is
\begin{equation}\label{eqConcInterp}
\p_t u_\lambda +\lambda(u_\lambda\cdot\nabla)u_\lambda+(1-\lambda)(u_\lambda\cdot\nabla)w-\Lap u_\lambda+\nabla p=0,\quad \nabla\cdot u_\lambda=0,
\end{equation}
where $w$ denotes the solution of the model system \re{eqBurgersW1}, with initial data $w_0=u(0)$. 
Now the proof of the maximum principle in Lemma \ref{lemMP} fails for $\lambda>0$ in this case. However, using the bounds on $w\in L^\infty(\varepsilon,T;H^2)$, for example, an easy argument shows that for any $T>0$ there exists $\lambda_0\in(0,1)$ such that, if $u_\lambda$ is a (local) classical solution of \re{eqConcInterp}, then $\|u_\lambda\|_1$ remains bounded on $[0,T]$ for any $\lambda<\lambda_0$. That is, we can interpolate between the minimum existence time for such solutions of the model system and the Navier--Stokes equations, using $\lambda$. 

More careful choices for the interpolants $A_\lambda$ could lead to new insights about the Navier--Stokes equations, within families of related systems.  For example, if one could prove global well-posedness results for all sufficiently small $\lambda>0$, the behaviour of solutions at critical values of $\lambda>0$ would be of great potential interest.

Finally, we remark that an investigation of regularity criteria for the Navier--Stokes equations in the magnetization variables formulation (\ref{eqNSEW1}--\ref{eqNSEW2}) may be worthwhile, particularly given the apparent similarity with advection-diffusion systems (as discussed in \cite{Friedlander_Vicol_2011,Silvestre_Vicol_2012}, for example). Indeed, it would be natural to study the $L^\infty(0,T;L^3)$ endpoint of Serrin's regularity condition in this formulation, which might yield an alternative approach to the celebrated result of Escauriaza, Seregin and \v Sver\'ak \cite{Escauriaza_S_S_2003}.  

\section*{Acknowledgements}
We would like to thank James Robinson for his thoughtful advice.

BCP was partially supported by an EPSRC Doctoral Training Award and partially by postdoctoral funding from ERC 616797.

The final version of this article will appear with DOI:
10.1016/j.jde.2017.12.036

\end{document}